\documentclass[10pt, conference]{IEEEtran}
\usepackage[utf8]{inputenc}
\usepackage[russian,english]{babel}
\usepackage{amssymb}
\usepackage{amsthm,amsmath}
\usepackage{graphicx}
\usepackage{amsfonts}
\usepackage{amssymb,latexsym}
\usepackage{url}
\usepackage{amsmath}
\usepackage{mathtools}
\usepackage{ifthen}
\newtheorem{theorem}{Theorem}
\newtheorem{definition}[theorem]{Definition}
\newtheorem{example}[theorem]{Example}
\newtheorem{remark}[theorem]{Remark}

\newtheorem{corollary}[theorem]{Corollary}
\newtheorem{proposition}[theorem]{Proposition}
\newtheorem{lemma}[theorem]{Lemma}

\def\d{{\rm dom}\hphantom{.}}

\def\2{{\bf 2}}
\def\P2{{\rm Par}(\2)}

\def\OA{{\rm Op}(\k)}
\def\O2{{\rm Op}(\2)}
\def\Pn2{{\rm Par}^{(n)}(2)}
\def\PA{{\rm Par}(\k)}
\def\pp{{\rm pPol}\,}
\def\PT{{\rm Pol}\,}
\def\N{\mathds{N}}
\def\st{~|~}

\def\iff{\Longleftrightarrow}

\def\vv{\lower 4pt \hbox{$\buildrel
{\textstyle{v}}\over{\scriptstyle{\sim}}$}}
\def\uu{\lower 4pt \hbox{$\buildrel
{\textstyle{u}}\over{\scriptstyle{\sim}}$}}
\def\ww{\lower 4pt \hbox{$\buildrel
{\textstyle{w}}\over{\scriptstyle{\sim}}$}}
\def\r{\rho}
\def\F{{\mathcal F}}
\def\3{{\bf 3}}
\def\k{{\bf k}}
\def\JA{{\rm J}(\k)}
\def\J2{{\rm J}(\2)}
\def\N{{\Bbb N}}

\def\f{\phi}
\def\dom{{\rm dom \;}}
\def\powerset{{\mathcal P}}
\DeclareMathOperator{\id}{id}

\def\PAl{{\Omega_{<\ell}(\k)}}

\def\uu{{\bf u}}
\def\vv{{\bf v}}
\def\ww{{\bf w}}

\def\0{{\bf 0}}
\def\1{{\bf 1}}
\def\image{{\rm img \;}}

\begin{document}
\title{Hereditarily Rigid Relations \\ {\large  Dedicated to Professor I.G.
Rosenberg on the occasion of his 80-th birthday.}}

\author{
\and
\IEEEauthorblockN{\hspace{1cm}Miguel Couceiro}%
\IEEEauthorblockA{\hspace{1cm}LORIA \\
\hspace{1cm}(CNRS - Inria Nancy G.E. -Univ. Lorraine)\\
\hspace{1cm}Vand{\oe}uvre-l\`es-Nancy,  France\\
\hspace{1cm}Email: miguel.couceiro@inria.fr}
\and
\IEEEauthorblockN{\hspace{3cm}}
\and
\IEEEauthorblockN{Lucien Haddad}%
\IEEEauthorblockA{Dept. of Mathematics \& CS\\
  Royal Military College of Canada\\
Kingston, Ontario, Canada\\  Email: haddad-l@rmc.ca}      \and
        \IEEEauthorblockN{\hspace{2cm}}
                         \and
        \IEEEauthorblockN{\hspace{2cm}}
       \and
                 \IEEEauthorblockN{Maurice Pouzet \hspace{1cm}}%
\IEEEauthorblockA{ICJ, Univ. Claude-Bernard Lyon 1  \hspace{1cm}\\
 Villeurbanne, France  \hspace{1cm}\\
 and Dept. of Math \& Stat,  \hspace{1cm}
 \\Univ. Calgary, Alberta, Canada  \hspace{1cm}\\
  Email: pouzet@univ-lyon1.fr  \hspace{1cm}}
 \and
\IEEEauthorblockN{\hspace{2cm}}
\and
      \IEEEauthorblockN{Karsten Schölzel \hspace{2cm}}%
\IEEEauthorblockA{University of Luxembourg  \hspace{2cm}\\
Mathematics Research Unit  \hspace{2cm}\\
Grand Duchy of Luxembourg \hspace{2cm}\\
Email: karsten.schoelzel@uni.lu \hspace{2cm}}}
  \maketitle

\begin{abstract} An $h$-ary relation $\r$ on a finite set $A$ is said to be \emph{hereditarily rigid} if the  unary partial functions on $A$ that preserve $\r$ are the subfunctions of the identity map or of constant maps. A family of relations ${\mathcal  F}$ is said to be \emph{hereditarily  strongly rigid} if the partial functions on $A$ that preserve  every $\r \in {\mathcal  F}$ are the subfunctions of projections or constant functions.  In this paper we show that hereditarily rigid relations exist and we give a lower bound on their arities. We also prove that no finite hereditarily  strongly rigid  families of relations exist and we also construct an infinite   hereditarily  strongly rigid  family of relations.

\end{abstract}

\baselineskip=13.5pt

\section{Introduction}

Let $k \ge 2$ and  $\k:=\{0,\dots,k-1\}$.   For a positive integer
$n$, an  $n$-ary {\it partial function}  on $\k$ is a map $f:
\dom (f) \to \k$ where $\dom (f)$ is a subset of $\k^n$, called the
{\it domain} of $f$. Let $\PA^{(n)}$ denote the set of all $n$-ary
partial functions on $\k$ and let ${\PA} := \bigcup\limits_{n\ge 1}
\PA^{(n)}$. Set $\OA^{(n)} := \left\{f \in \PA^{(n)} \st
\dom (f) = \k^n \right\}$ and call $ \OA := \bigcup\limits_{n\ge 1}
\OA^{(n)}$ the set of all  ({\it total}) {\it functions} on $\k$.

A partial function $f \in \PA^{(n)}$ is
a \emph{subfunction} of $g \in \PA^{(n)}$ (in symbols $f \le g$ ) if $\dom(f)
\subseteq \dom(g)$ and $f(\vec{x})=g(\vec{x})$ for all $\vec{x} \in \dom(f)$.
A partial function $f$ is \emph{constant} if it does not have two distinct values.
\\
For every positive integer $n$ and every $1 \le i \le n$, let
$e_{i}^n$ denote the $n$-ary \emph{ $i$-th  projection} defined
by
 $e^n_{i} (\vec{a}) = a_i$ for all
$\vec{a}:=(a_1, \dots, a_n) \in \k^n$. Set $\JA=\left \{e^n_{i} \st
n \in \N,   1
\leq i \leq n \right \}$, i.e., $\JA$ is the set of all projection functions on $\k$. We denote by $\id$ the identity map on $\bf k$ (instead of $e_1^{1}$). 
Any subfunction of a projection is called a \emph{partial projection}.
\\
As usual,  if $f$ is any function we denote by $\dom (f)$ its domain and by $\image (f)$ its range. If $f$ and $g$ are two functions and $\image (g)\subseteq \dom (f)$ we set $f\circ g$ for the composition defined in a natural way.
A {\it partial clone} on $\k$ is a composition closed subset
of $\PA$ containing the set of all projections.  A partial clone $C$  contained in the
set $\OA$ of all total functions is called {\it a clone} on $\k$. Moreover, a partial clone $C$ is called {\it strong} if it contains all subfunctions of its functions, i.e., if for all functions $f$
and $g$, if  $f \in C$ and $ g \le f$, then $g \in C$.\\
Let $m\ge 1$, an  $m$-ary relation on $\k$ is a subset $\r$ of
$\k^m$.  Let $n\geq 1$ and let $f\in \PA^{(n)}$. We say that
$f~preserves~ \r$, or $\r$ {\it is invariant under} $f$,
if for every $m \times n$ matrix $M=[M_{ij}]$ whose columns $M_{*j}
\in \r,~(j=1,\ldots n)$ and whose rows $M_{i*} \in \dom (f)$
$(i=1,\ldots,m)$,  we have $(f(M_{1*}),\ldots,f(M_{m*})) \in
\r$.
\\
\noindent  Set $\pp \r:=\{f \in \PA \st f~{\rm
preserves}~\r\}$ and \break $\PT \r:= (\pp \r )\cap
\OA$.

Moreover  let
$\pp  ^{(1)}\rho  := (\pp \rho) \cap \PA ^{(1)}$ denote all partial unary functions that preserve $\r$. Similarly let
$\PT^{(1)}  \rho  := (\PT \rho) \cap \OA ^{(1)}$ denotes all (total)  unary functions that preserve $\r$.

\begin{remark} It follows from the definition of $\pp \r$ that if there is no matrix $M$ whose lines and columns  satisfy the two
conditions above, then $f \in \pp\r$.
\end{remark}

\begin{remark}
Given a non-empty $h$-ary relation $\r$ on $\bf k$ then, 
provided that $k\geq 2$, there is always a partial constant unary function $f$ that is not a partial projection function and that preserves $\r$. 
Indeed if for some $a \in \k$, $(a,\dots,a) \notin \r$, then choose $b \ne a$ and define the partial constant function $f_{b}$ by $\dom(f_b)=\{a\}$ and $f_b(a)=b$. 
Then $f_b \in \pp  ^{(1)} \rho$  by definition. On the other hand, if $(a,\dots,a) \in \r $ for all $a \in \k$, then again the above defined partial constant function 
$f_b$ satisfies $f_b \in \pp ^{(1)}  \rho$.
\end{remark}

It is well known (see, e.g.,  \cite{lau2006}  Sections 2.6 and 20.3) and easy
to show that $\pp\r$ (resp. $\PT \r$) is a strong partial clone
on $\k$ (resp. a clone on $\k$) called the partial clone (resp. the clone) {\it
determined by $\r$}.

\begin{remark} 
In what follows, we deal only with {\it non-empty} relations.  
\end{remark}

Rigid binary relations are introduced in \cite{v-p-h}. An $h$-ary relation $\r$ on $\k$ is said to be {\em{rigid}} if $\PT ^{(1)}\r= \{\id\}$, i.e., if the identity map is the only unary function on $\k$ that preserves $\r$. Moreover $\r$ is {\em strongly rigid} if $\PT \r = \JA$, i.e., if the set of all (total) functions on $\k$ that preserve $\r$ consists of all projection functions on $\k$.
Rigid and strongly rigid relations  have been studied in the literature (see, e.g.
\cite{anne,lan-pos, benoit-claude,ivo, v-p-h}). 

When dealing with partial functions, it seems  natural to extend this concept. However, as mentioned in Remark 2,  rigid relations  with respect to partial functions are trivial. 
In fact, since $\pp\emptyset= \PA $, for  $k\geq 2$, such relations do not exist at all. 
Hence, we focus on the notion of semirigidity. 

Recall that a relation $\rho$ is \emph{semirigid} (resp. \emph{strongly semirigid}) if every unary function (resp. arbitrary function) that preserves $\rho$ is the identity map or a constant map 
(resp. a projection or a constant function) (see, e.g. \cite{miyakawa2, lan-pos, miyakawa1,zadori}). 
We call these generalizations \emph{hereditarily rigid} and \emph{hereditarily strongly rigid} (rather than hereditarily semirigid  and hereditarily strongly semirigid).
For ${\ell} \geq 1$, we define 
\begin{eqnarray*}
\PAl &:=& \{ f \in \PA ^{(1)} \mid f \leq \id \}  \\
 &\cup& \{ g \in \PA ^{(1)} \mid |\image (g)| < {\ell} \}. 
\end{eqnarray*}

\begin{definition}
  {\rm Let $h \ge 1$, $\rho$ be an $h$-ary relation on $\k$ and $\ell$ be an integer with $1 \leq {\ell} \leq k$.
  Then $\rho$ is called \emph{hereditarily $\ell$-rigid} if
  $\pp ^{(1)} \rho = \PAl$. }
    \end{definition}

For simplicity,  we have consider this condition instead $\pp ^{(1)}\rho\subseteq \PAl$ which would seem the natural generalization of semirigidity. 
We leave to a further study, the consideration of this generalization and we refer to Lemma \ref {arity} for a possible  relationship between the two notions.

By  Remark 2,   we know that  for any relation $\r$ on  at least two elements,  there is always  a partial constant function that preserves $\r$. In particular:

\begin{proposition}\label{norigid}
 Let $k \ge 2$. Then there is no hereditarily $1$-rigid relation on $\k$.
\end{proposition}

 For $\ell  =2$, we will drop $\ell$ in the definition of $\ell$-rigidity.


We define {\it hereditarily   strongly rigid relations} as follows.

  \begin{definition}
   {\rm A family of relations ${\mathcal R}$ on  $\k$ is said to be {\em hereditarily strongly rigid} if
$\displaystyle \bigcap_{\r \in {\mathcal R}} \pp(\r)$ is the partial clone generated by all partial constant functions on $\k$}.
 \end{definition}

In particular an $h$-ary relation $\r$  on $\k$ is {\em{hereditarily  strongly rigid}} if $\pp \rho$ is the  partial clone on $\k$ generated by all constant functions on $\k$. Equivalently a partial function preserves $\r$ iff it is a partial projection or a partial constant function.

Since a hereditarily $\ell$-rigid relation ($\ell\geq 2)$ (resp.  a  hereditarily strongly rigid relation)  is preserved by the partial constant functions and is non-empty, then it  must contain the diagonal relation $\{(x,\dots,x) \st x \in \k\}$.

\section{Hereditarily rigid relations}

In this section we show that non-trivial hereditarily rigid relations exist and we give an  upper bound on the size of their  domain.

Let $X$ be a set   and $m\geq 1$. We denote by $\powerset (X)$ the powerset of $X$. We denote by ${X \choose m}$ the subset of $\powerset (X)$ made of $m$-element subsets of $X$.
We recall that an \emph{antichain} of subsets of $X$ is a collection of subsets such that none is contained in another.
We recall the famous theorem of Sperner (see \cite{engel}):

\begin{theorem}
  Let $n$ be a non-negative integer.  The largest size  of an  antichain family of subsets of  an $n$-element set  $X$ is $\binom{n}{\lfloor n/2 \rfloor}$. It is only realized by
 $\binom{X}{\lfloor n/2 \rfloor}$ and $\binom{X}{\lceil n/2 \rceil}$.
\end{theorem}

Let $n\geq 1$.  We set $[n]:=\{1, \dots n\}$; we identify $n$-tuples of elements of $X$ with maps  from $[n]$  to $X$.  Let $\vec{x}:=(x_1, \dots, x_n)\in X^n$; considering it as a map,  $\image(\vec{x})= \{x_1, \dots, x_{n}\}$; we set $\vert \vec{x} \vert := \vert \image(\vec{x})\vert$.
Suppose  $1\leq m\leq n$. We set $\beta_{m}^{n}(X):= \{\vec{x}\in X^n: \vert \vec{x}\vert =m\}$ and  $\beta_{<m}^{n}(X):=\bigcup_{m'<m}\beta_{m'}^{n}(X)$.
We note that $\beta^{n}_{n}(X)$ is the set of all $n$-tuples with pairwise distinct entries, hence if $\ell:=\vert X\vert$, then $\vert \beta^{n}_{n}(X)\vert$ is the falling factorial $\ell^{\underline n}:= \ell\cdot (\ell-1)\cdots (\ell-n+1)$. On the other hand, $\beta^{n}_{\ell}(X)$ identifies with the set  of surjective maps from $[n]$ onto $X$. Denote this number by $s(n,\ell)$ and recall that it satisfies the formula:

\[ s(n,\ell) = \sum_{j=1}^{\ell} (-1)^{\ell-j} \binom{\ell}{j} j^n. \]

\begin{lemma}\label{arity} Let $\r$ be an $h$-ary relation on $\k$, $\ell\geq 1$, $Orb_{<\ell}(\rho):= \{x\circ i: i \in \beta_{<\ell}^{h}(\k)\cap \rho \; \text{and}\;  x\in \PA ^{(1)}\}$,  and $\hat{\rho}:= \rho \cup Orb_{<\ell}(\rho)$.  Then
\begin{enumerate}
\item $ \PAl\cup \pp^{(1)} \rho\subseteq \pp^{(1)}\hat {\rho}$;
\item If $\rho$ is hereditarily $\ell$-rigid, then $\hat{\rho}= \rho$;
\item If  $\hat \rho$ is hereditarily $\ell$-rigid, then  $\pp^{(1)} \rho\subseteq  \PAl$;
\item If $h<\ell$ and $\PAl \subseteq \pp^{(1)} \rho$, then $\pp^{(1)} \rho=\PA ^{(1)}$.
\end{enumerate}
\end{lemma}
 \begin{proof} For (1),  observe first that  $f\circ x\circ i \in Orb_{<\ell}(\rho)$ for every $f\in \PA ^{(1)}$,  $x\circ i \in Orb_{<\ell}(\rho)$ such that $\image (f) \subseteq x\circ i$; next  prove successively that  $\PAl \subseteq \pp^{(1)}\hat {\rho}$ and  $\pp^{(1)} \rho\subseteq \pp^{(1)}\hat {\rho}$.     (2) is   immediate. (3) follows from (1).  For (4) observe that a  partial function $f$ belongs to $\pp^{(1)} \rho$  if and only if every subfunction with domain of size at most $h$ belongs to $\pp^{(1)} \rho$.
 \end{proof}

 \begin{corollary}\label{arity}
 Let $k\ge 2$ and $\ell$, $1\leq \ell\leq k$. If  an $h$-ary relation $\rho$ on $\k$ is hereditarily $\ell$-rigid, then $h\geq \ell$.
 \end{corollary}

 Let $\Psi_{\ell}(\k)$ be the set of one-to-one partial unary functions on   $\bf k$ which are not below the identity and whose domain has size $\ell$. 
 This definition amounts to:
 $$\Psi_{\ell}(\k)= \{ f \in \PA ^{(1)} \mid |\dom (f)| = {\ell}\; \text{and} \;  f \notin \PAl \}.$$

 \begin{lemma}
 Let $\rho$ be a  hereditarily $\ell$-rigid relation and $f \in \Psi_{\ell}(\k)$.
Then there is some $\vec{x} \in \rho$ with
  $\image(\vec{x}) = \dom (f)$.
\end{lemma}

\begin{proof}  Clearly,  $f\not \in \pp^{(1)}\rho$. Hence, there is  
$$\vec{x}:= (x_1,\dots, x_h)\in \rho\cap (\dom(f))^h$$ with $f(\vec{x})\not \in \rho$. Set $X:=\image(\vec{x})$. 
Then $g:= f_{\restriction X}\not \in  \pp^{(1)}\rho$. Since $\pp^{(1)}\rho= \PAl $, $\vert \image(g)\vert=\ell$ and thus $X=\dom(f)$.
\end{proof}

\medskip
The next lemma shows that we only need to consider functions whose domain has size $\ell$.

\begin{lemma}\label{lem:key}
  Let $f \in \PA ^{(1)}$ and ${\ell} \geq 1$ such that $g \in \PAl$ for all $g \leq f$ with $|\dom (g)| = {\ell}$.
  Then $f \in \PAl$.
\end{lemma}

\begin{proof} Let $f \in \PA ^{(1)}$. If  $\vert \dom(f)\vert <\ell$, then $\vert \image(f)\vert <\ell$ hence $f\in \Omega_{<\ell}(\bf k)$. On the other hand if $\vert \dom(f)\vert =\ell$, then, since $f\leq f$,  $f\in \Omega_{<\ell}(\bf k)$. So let  $|\dom (f)| > \ell$.  Assume to the contrary that $f \notin \PAl$, i.e. $f \not\leq \id$ and $|\image (f)| \geq \ell$. Then there is some $X \subseteq \dom (f)$ with $|X| = {\ell}$ and $|f(X)| = {\ell}$.
 By hypothesis, $g:=f_{\restriction X}\in \Omega_{\ell}(\bf k)$. Since $\vert \image(g)\vert =\ell$, we have $g\leq \id $, i.e., $f(x)=x$ for all $x\in X$.

  (1) Suppose that there exists $y \in \dom (f) \setminus X$ such that   $f(y) \in X$. Then consider $Y := X \setminus \{ f(y) \} \cup \{y\}$. Let $g := f_{|Y}$.
      Clearly, $g\leq f$ and $|\dom (g)| = |Y| = \vert X\vert ={\ell}$, hence  from the hypothesis of the lemma   we have $g \in \PAl$. Since  $\vert  \image (g) \vert =  |X| = {\ell}$ we have $g \leq \id$, i.e.,
      $f(y) = y$, but this contradicts $y \notin X$ and $f(y) \in X$.

   (2)   Otherwise, we have  $f(y) \notin X \cup \{y\}$ for all $y \in \dom (f) \setminus X$. Since $f\not \leq \id$ there exists some $y \in \dom (f)\setminus X$ such that $f(y)\not \in X\cup \{y\}$. Then consider $Y := (X \setminus \{ x \}) \cup \{y\}$ for some $x \in X$ and let $g:= f_{|Y}$.
      Clearly, $g\leq f$ and $|\dom (g)| = |Y| = \vert X\vert ={\ell}$ thus,  again,   $g \leq \id$, hence       $f(y) = y$, but this contradicts $f(y) \notin X \cup \{y\}$.

In both cases, we get a contradiction and, hence, $f \in \PAl$ as claimed.
\end{proof}

\begin{theorem}\label{cor:psi}
A relation $\rho$ is hereditarily $\ell$-rigid iff
  $ \pp^{(1)}\r \cap \Psi_{\ell}(\k) = \emptyset$ and $\PAl \subseteq  \pp^{(1)}\r$.
\end{theorem}
\begin{proof} Trivially, if $\rho$ is hereditarily $\ell$-rigid, then the two conditions are satisfied since $\pp^{(1)}\r=\PAl$. Conversely, let $f\in  \pp^{(1)}\r$. If $f\not \in \PAl$, then,  according to Lemma \ref{lem:key}, there is some $g\le f$  with $|\dom (g)| = {\ell}$ such that $g \not \in \PAl$. We have  $g\in \Psi_{\ell}(\k)$ hence, according to the first condition,  $g \not \in \pp^{(1)}\r$.  This  is impossible since every subfunction of $f$ must  belong  to  $\pp^{(1)}\r
$. This proves that $\pp^{(1)}\r\subseteq \PAl$. Since the second condition asserts that  $\PAl \subseteq  \pp^{(1)}\r$, we have $\pp^{(1)}\r=\PAl$, that is,  $\rho$ is hereditarily $\ell$-rigid. \end{proof}

%
%


It follows from the statements above that we only need to consider the tuples with exactly $\ell$ different entries in a hereditarily  $\ell$-rigid relation $\rho$.
We connect $\ell$-rigid relations to antichains of subsets.

\begin{definition}
  Let $\rho$ be an $h$-ary relation on $\k$, and ${\ell} \geq 1$. We define the function  $T_{\rho}^{\ell}$ from $\beta_{\ell}^{\ell}(\bf k)$ to $\powerset (\beta_{\ell}^{h}([\ell]))$  by setting:
 \begin{eqnarray*}
  \rho^{\ell}((x_1, \dots, x_{\ell})) &=&\\
  \{ (i_1, \dots, i_h)&\in& \beta_{\ell}^h([\ell]) \mid (x_{i_1},\dots,x_{i_h}) \in \rho \}.
 \end{eqnarray*}

 \end{definition}

     Identifying $t$-tuples  with functions, and denoting by $\circ$ the composition of functions, the definition above rewrites as:  $$T_\rho^{\ell}(\vec{x}) := \{ \vec{i} \in \beta_{\ell}^h ([\ell])\mid \vec{x}\circ \vec{i}\in \rho \}.$$

In Sections 2 and 4 we will use the following two definitions and Proposition \ref{proposition:compability} below.

\begin{definition}
  Let $T\colon  \beta_{\ell}^{\ell}(\k)\rightarrow  {\powerset}({ \beta_{\ell}^h([\ell]) )}$. We define the function $\hat{T}$ from $\beta_{\ell}^{\ell}(\k)$to ${\powerset}({ \beta_{\ell}^h(\k)) }$ by setting:
    $$\hat{T}(\vec{x}) := \{ \vec{x}\circ \vec{i}\mid \vec{i} \in T(\vec{x}) \}.$$
\end{definition}

\begin{definition}\label{def:12}
  Let $T\colon  \beta_{\ell}^{\ell}(\k)\rightarrow {\powerset}({ \beta_{\ell}^h([\ell])) }$. We define the $h$-ary relation $\rho_T$ on $\bf k$ by setting:
  \[ \rho_T :=   \beta^{h}_{<\ell}(\k) \cup \bigcup_{\vec{x} \in \beta_{\ell}^{\ell}(\k)} \hat{T}(\vec{x}). \]
\end{definition}
\begin{proposition} \label{proposition:compability}
  Let $\rho$ be an $h$-ary relation on $\k$, $\pi$ be a permutation on $[\ell]$  and $\vec{x}  \in \beta_{\ell}^{\ell}(\k)$.
  Then
  \[\vec{i} \in T^{\ell}_{\rho}(\vec{x}\circ \pi ) \; \text{if and only if}\;   \pi\circ \vec{i} \in T^{\ell}_\rho(\vec{x}). \]
\end{proposition}

\begin{proof} The fact that   $\vec{i}\in T_{\rho}^{\ell}(\vec{x}\circ \pi)$ amounts to  $(\vec{x}\circ \pi)\circ \vec{i}\in \rho$, whereas 
$\pi\circ \vec{i} \in T^{\ell}_\rho(\vec{x})$   means that $\vec{x}\circ (\pi\circ \vec{i})\in \rho$. \end{proof}

Members of  $\beta_{\ell}^{\ell}(\k)$ are one-to-one maps hence, as partial functions they are invertible on their image. If  $\vec{x} \in \beta_{\ell}^{\ell}(\k)$ we denote by $\vec{x}^{-1}$ its inverse.
\begin{definition}
  Let $\vec{x},\vec{y} \in \beta_{\ell}^{\ell}(\k)$. We set  $F_{\vec{x}\to\vec{y}}:= \vec{y}\circ \vec{x}^{-1}$.  \end{definition}
 According to this definition  $F_{\vec{x}\to\vec{y}}$ is  a partial unary function on $\bf k$ with domain $\image(\vec{x})$ and image $\image(\vec{y})$.

  \begin{lemma}\label{lem:key}
Let $\rho$ be an $h$-ary relation on $\k$, $\ell\geq 2$  and  $\vec{x},\vec{y} \in \beta_{\ell}^{\ell}(\k)$.
Then $F_{\vec{x}\to\vec{y}} \in \pp^{(1)} \rho$ implies $T^{\ell}_\rho(\vec{x}) \subseteq T^{\ell}_\rho(\vec{y})$. The converse holds provided that  $\Psi_m(\k) \subseteq  \pp^{(1)}\r$ for all $m < \ell$.

\end{lemma}

\begin{proof} By definition, we have $F_{\vec{x}\to\vec{y}}\in \pp^{(1)} \rho$ iff $F_{\vec{x}\to \vec{y}}\circ \vec{u}\in \rho$ for every $\vec{u}\in \rho\cap \image(\vec {x})^h$.
Suppose that $F_{\vec{x}\to\vec{y}} \in \pp^{(1)} \rho$. We show that $T_{\rho}^{\ell}(\vec{x})\subseteq T_{\rho}^{\ell}(\vec{y})$ holds. Indeed, let $\vec{i}\in T_{\rho}^{\ell}(\vec{x})$. Then $\vec{x}\circ \vec{i}\in \rho$ thus $F_{\vec{x}\to \vec{y}}\circ \vec{x}\circ \vec{i}=\vec{y}\circ{i} \in \rho$, proving that $\vec{i}\in T_{\rho}^{\ell}(\vec{y})$.

Conversely, suppose that  $T_{\rho}^{\ell}(\vec{x})\subseteq T_{\rho}^{\ell}(\vec{y})$ holds. Let $\vec{u}\in \rho\cap \image(\vec {x})^h$. Set $m:=\vert \image(\vec {u})\vert $. If $m<\ell$, then since  $\Psi _m(\bf k)\subseteq \pp\rho^{(1)}$ it follows that $F_{\vec{x}\to \vec{y}}\circ \vec{u}\in \rho$.
If $m=\ell$ set $\vec{i}:= \vec{x}^{-1}\circ \vec{u}$. Then $\vec{x}\circ \vec{i}= \vec{u}\in \rho$ hence  $\vec{i}\in T_{\rho}^{\ell}(\vec{x})$. Since $T_{\rho}^{\ell}(\vec{x})\subseteq T_{\rho}^{\ell}(\vec{y})$, then $\vec{i}\in T_{\rho}^{\ell}(\vec{y})$ hence $F_{\vec{x}\to\vec{y}}\circ \vec{u}=\vec{y}\circ{i} \in \rho$.  And thus $F_{\vec{x}\to\vec{y}}\in \pp \rho^{(1)}$.
\end{proof}

\medskip
From this,  we obtain:
\begin{theorem} \label{corollary:rigidisantichain}
  If a  relation $\rho$ on $\k$ is hereditarily $\ell$-rigid, then  ${\mathcal  T} := \{ T^{\ell}_\rho(\vec{x}) \mid \vec{x} \in \beta_{\ell}^{\ell}(\k) \}$ is an antichain with respect to set inclusion. The converse holds provided that $\PAl\subseteq \pp^{(1)} \rho$.
\end{theorem}
\begin{proof}  If $\PAl\subseteq \pp^{(1)} \rho$, then $\Psi_m(\k) \subseteq  \pp^{(1)}\r$ for all $m < \ell$. Now, if    $\rho$  is hereditarily $\ell$-rigid, then $\PAl= \pp^{(1)} \rho$. Thus the equivalence in Lemma \ref{lem:key} holds in the cases we are considering and gives the result. For example, let  $T^{\ell}_\rho(\vec{x})$ and $T^{\ell}_\rho(\vec{y})\in \mathcal T$. Suppose that $T^{\ell}_\rho(\vec{x}) \subseteq T^{\ell}_\rho(\vec{y})$. Applying Lemma \ref{lem:key} we get that the map  $F_{\vec{x}\to\vec{y}} \in \pp  \rho^{(1)}$. If  $\rho$ is hereditarily $\ell$-rigid, this map must be the identity, i.e.,  $x=y$ hence $T^{\ell}_\rho(\vec{x})=T^{\ell}_\rho(\vec{y})$. This proves that if $\rho$ is hereditarily $\ell$-rigid, then  $\mathcal T$ is an antichain. The converse follows the same line. \end{proof}


Since   $T_{\rho}^{\ell}(\vec{x})$ is a subset of $\beta_{\ell}^{h}([\ell])$ whose cardinality is the number $s(h, \ell)$  of surjections of $[h]$ onto $[\ell]$, then with the help of Sperner's theorem we obtain:

\begin{corollary}
  Let $\rho$ be an $h$-ary hereditarily $\ell$-rigid relation on $\k$. Then $k^{\underline \ell} \leq {s(h,\ell)\choose s(h,\ell)/2}$ where $k^{\underline \ell}=k\cdot(k-1)\cdots (k-\ell+1)$ and $s(h, \ell)$ is  the number of surjections of $[h]$ onto $[\ell]$.

\end{corollary}

The aim of the next  section  is to show that for $\ell=2$, this bound can be attained, and that for $\ell > 2$
the actual upper bound is not much lower than the one in the previous corollary.

\section{The case $\ell=2$}

For $h\geq 2$, we have $s(h,2) = 2^h-2$, and thus $s(h,2)/2 = 2^{h-1}-1$ is odd. Hence,  there is just one antichain of maximal size on the set $\beta_{2}^{h}([2])$, 
namely ${\mathcal   X}_{2}^h := \binom{\beta_{2}^h([2])}{2^{h-1}-1}$. Therefore, for any $X \in {\mathcal   X}_2^h$ the dual set $X^{\mathrm{d}}$ (interchanging 0's and 1's) is different from $X$, but also belongs to ${\mathcal   X}_2^h$.
Let $T\colon  \beta_2^2(\k) \to {\powerset}({ \beta_2^h([2])) }$ be an injective function such that
$T(x,y) \in {\mathcal   X}_2^h$ and $T(y,x) = T(x,y)^{\mathrm{d}}$.
Then $T$ fulfills the condition  in Proposition \ref{proposition:compability}.  The relation $\rho_T$ given by Definition \ref{def:12} satisfies the condition in Theorem \ref{corollary:rigidisantichain}, 
and thus it is hereditarily $2$-rigid. From this observation we get:

\begin{theorem}
  There exists an $h$-ary hereditarily $2$-rigid relation on $\k$ if and only if
  \[ k\cdot(k-1) \leq \binom{2^h-2}{2^{h-1}-1}. \]
\end{theorem}

For example, if $h:=1,2,3,4, 5$ the largest values  of $k$ are respectively: $0,2,5,59, 12455$.
Using an approximation of the binomial coefficient via Stirling's formula we get an  $h$-ary hereditarily $2$-rigid relation iff
\[ k \lesssim \sqrt{\frac{4^{2^{h-1}-1}}{\sqrt{\pi(2^{h-1}-1)}}} = \frac{2^{2^{h-1}-1}}{\sqrt[4]{\pi(2^{h-1}-1)}}, \]
which  basically grows double exponentially.

\section{The case ${\ell} \geq 3$}

Denote by $\mathfrak S_\ell$ the set of permutations of $[\ell]$. Let $\pi\in \mathfrak S_\ell$ and $X\subseteq \beta_{\ell}^{h}([\ell])$. We set $\pi(X):=\{\pi\circ \vec{i}: \vec{i}\in X\}$.
If  ${\ell} \geq 3$, a function $T$ constructed from the antichain ${\mathcal   X}_{\ell}^h$ similar to the case $\ell=2$,
does not automatically fulfill the condition in Proposition  \ref{proposition:compability},
because there can be some element $X \in {\mathcal   X}_{\ell}^h$, and some non-identical permutation $\pi$ on $[\ell]$,
with $\pi(X) = X$. The element $\vec{x} \in \beta_{\ell}^{\ell}(\k)$ that would be assigned to this element by $T$ would then
have the property that $F_{\vec{x}\to\vec{x}\circ \pi} \in \pp^{(1)}\rho_T$, i.e., $\rho_T$ would  not be hereditarily $\ell$-rigid.

For an arbitrary $\vec{y} \in \beta_{\ell}^h([\ell])$, let $Y := \{ \pi\circ \vec{y} \mid \pi \in \mathfrak S_\ell \}$ and $Y':= \{\{y\}: y\in Y\}$.
Then $|Y'| = {\ell}!$. Let ${\mathcal   Y}_{\ell}^h$ be an  antichain of maximal size in  $\powerset (\beta_{\ell}^h([\ell]) \setminus Y)$
and construct $T\colon  \beta_{\ell}^{\ell}(\k) \to {\powerset }({ \beta_{\ell}^h([\ell]) )}$ as follows.

Let $T_1\colon  \beta_{\ell}^{\ell}(\k) \to {\mathcal   Y}_{\ell}^h$ be an injective function
such that  $T_1(\vec{x}\circ \pi) = \pi^{-1} T_1(\vec{x})$.
Let $T_2\colon \beta_{\ell}^{\ell}(\k) \to Y'$ with
$T_2(\vec{x}\circ\pi) = \pi^{-1} T_2(\vec{x})$, and  set $T(\vec{x}) := T_1(\vec{x}) \cup   T_2(\vec{x})$.

This function $T$ fulfills the condition in Proposition \ref{proposition:compability}, and its range is an antichain. By Theorem \ref{corollary:rigidisantichain},
the relation $\rho_T$ is hereditarily $\ell$-rigid.

\begin{corollary}
  Let $k$, $\ell$, and $h$ with $\ell < h$ such that
  \[ k^{\underline{\ell} }\leq \binom{s(h,\ell)-\ell !}{(s(h,\ell)-\ell !)/2}. \]
  Then there is an $h$-ary hereditarily $\ell$-rigid relation on $\k$.
\end{corollary}

This upper bound is not optimal, as the construction above is brute force. But there is only a constant (for constant $\ell$)
factor between this bound, and the one given before.

\begin{theorem}
  Let $r(\ell,h)$ be the maximum cardinality  of $h$-ary hereditarily $\ell$-rigid relations.
   Then
  \[ \binom{s(h,\ell)-\ell !}{(s(h,\ell)-\ell !)/2} \leq r(\ell,h) \leq \binom{s(h,\ell)}{(s(h,\ell))/2}. \]
  Furthermore, lower and upper bound on $r(\ell,h)$ differ approximately by a constant factor for constant $\ell$, and $h \gg \ell$:
   \begin{eqnarray*}
 \binom{s(h,\ell)-\ell !}{(s(h,\ell)-\ell !)/2} &\approx& \frac{2^{s(h,\ell)-\ell !}}{\sqrt{(\pi/2)(s(h,\ell)-\ell !)}}\\
 &\approx& \frac{1}{2^{\ell !}}\frac{2^{s(h,\ell)}}{\sqrt{(\pi/2)s(h,\ell)}} \\
 &\approx& \frac{1}{2^{\ell !}}\binom{s(h,\ell)}{(s(h,\ell))/2}.
   \end{eqnarray*}
\end{theorem}

\section{Hereditarily strongly rigid relations}

 In this section we prove that no finite hereditarily strongly rigid  family of relations exists and we also construct an infinite   hereditarily  strongly rigid  family of relations.
 
 \begin{lemma} \label{phin}  Let $k\geq 2$, $n \ge 3$ and  $n>h\geq 1$. There is an $n$-ary partial function $\phi_n$ on $\bf k$ that is neither a partial  projection nor a partial  
 constant and that preserves all $h$-ary relations on $\k$.
 \end{lemma}
 
 \begin{proof}
 For $n \ge 3$ consider the $n$-ary partial function $\phi_n$ defined by

 $\d(\f_n):=\{(0,1,1,\dots,1),(0,1,0,\dots,0),$ \hfill

 \hfill $(0,0,1,0,\dots,0),\dots,(0,0,\dots,0,1)\}$

 and

 $\f_n(0,1,\dots,1)=1$ and
 $\f_n(\vv)=0$ for $\vv \in \d(\f_n)$, $\vv \ne (0,1,\dots,1)$ . Thus we have

  \[
 ~\f_n\left(
    \begin{array}{cccccc}
      0 &1 & 1 &1 \dots &  1\\
       0 & 1 & 0  & 0\dots  &0\\
         0 &0 & 1 & 0\dots  &0 \\
         0 & 0 & 0 & 1 \dots  & 0\\
         \vdots & \vdots  & \vdots & \vdots & \vdots \\
         0 & 0 & 0  & 0\dots  &1\\
         \end{array} \right)
         =
         \left(
    \begin{array}{c}
      1 \\
        0 \\
         0 \\
         0\\
         \vdots \\
         0\\
         \end{array} \right)
 \]

Call $M$ the above left hand side matrix. The partial function $\f_n$ takes two values and so is not a constant function. 
Since the tuple $(1,0,\dots,0)^t$ is not a column of the matrix $M$, the partial function $\f_n$ is not a partial projection.
However, if $g \le \f_n$ is a subfunction of $\f_n$ with  $\d(g) \ne \d(\f_n)$, then $g$ is a partial projection function. 
Indeed, it is easy to see that if for some $i = 1,\dots,n$, we have $M_{i*}\not\in \d(g)$, then $g={e_i^{n}}_{|\d(g)}$.

 Now let $\r$ be  an $h$-ary relation on $\k$ with $h < n$ and let $N$ be an $h \times n$ matrix with all columns
$N_{*j} \in \r$ and all  rows $N_{i*} \in \d(\f_n)$. Since $h < n$ the matrix $N$ does not contain all rows of the matrix $M$ defined above, and so the partial function $\f_n$ restricted to the rows of the matrix $N$ is a partial projection function; it therefore preserves the relation $\r$.
\end{proof}

 \begin{corollary}
Let $k \ge 2$ and ${\mathcal   F} = \{\r_1,\dots,\r_t\}$ be a finite family of relations over $\k$. Then there is a partial function $\f$ that is neither a partial projection nor a partial constant function such that $\f \in \displaystyle \bigcap_{i=1}^t \pp(\r_i).$
\end{corollary}

\begin{proof}For $i=1,\dots,t$ let $h_i$ be the arity of the relation $\r_i$ and let
 $n:=\mathrm{max}\{h_i \st i=1,\dots,t\}+2$. Then the  $n$-ary function $\f_n$ constructed above preserves all relations $\r_i$.
\end{proof}

As a consequence of this, we get:

\begin{theorem}  Let $k \ge 2$.  Then there is no hereditarily strongly rigid finite family of relations  on $\k$.
\end{theorem}

\begin{remark}
Note that some relations that are strongly rigid with respect to total functions are studied and described in \cite{anne,lan-pos,benoit-claude,ivo}. 
\end{remark}

In view of the results above, one may ask if there exist a hereditarily strongly rigid infinite family of relations on $\k$, i.e., a family of relations ${\mathcal R}$ such that
$\bigcap_{\r \in {\mathcal R}} \pp(\r)$ is the partial clone generated by the constant functions on $\k$.

 In what follows we consider relations and partial clones on $\2:=\{0,1\}$. 
 The construction below can be generalized over any finite set.  
 Denote by ${\mathcal C}$ the strong partial clone generated by all partial constant functions on $\2$.

 \medskip
 
 {\bf Notation.}~ For $1 \le t < h $ define the $h$-tuple 
 $$\vv_t^h:=(\underbrace{1, \dots,1}_{{\rm t~times}},0,\dots,0)\in \2^h,$$
 $\Delta_t^h:= \2^h \setminus \{\vv_t^h\}$, and let ${\mathcal F}^{(h)}:= \{\Delta_1^h, \Delta_2^h,\dots,\Delta_{h-1}^h\}$.

Furthermore,  for $h \ge 2$, let
$\displaystyle \pp (\F^{(h)} ):= \bigcap_{j=1}^{h-1} \pp (\Delta_j^h)$.

 \medskip

 \begin{example}
  Consider $\F ^{(2)} = \{\{(0,0),(1,1),(1,0)\}\}$ and 
 $\F^{(3)}=\{\2^3\setminus \{(1,0,0)\},~\2^3 \setminus \{(1,1,0)\}\}$.
 \end{example}

 It is well known that polymorphisms of relations are strong partial clones, i.e., they contain all partial projections.  Moreover since $(0,\dots,0),(1,\dots,1) \in \Delta_t^h$ we deduce that $\pp(\Delta_t^h)$ contains all partial constant functions on $\2$. Thus for all $h \ge 2$ we have ${\mathcal C} \subseteq \pp(\F^{(h)})$.

 Note that if $\ww \in \2^h$ is any tuple with exactly $t$ symbols 1 and $h-t$ symbols 0, then $\pp(\Delta_t^h)= \pp (\2^h \setminus \{\ww\})$. 
 Indeed the relation $\Delta_t^h$ can be obtained from $\2^h \setminus \{\ww\}$ by some permutation of the variables of  $\2^h \setminus \{\ww\}$.

 \begin{example}
  Let $h=4$. Then $\vv_2^4 =(1,1,0,0)$ and 
  $\Delta_2^4 = \{0,1\}^4 \setminus \{(1,1,0,0)\}$.
 Take $\ww=(0,1,0,1)$ and $\lambda := \2^4 \setminus \{(0,1,0,1)\}$. It is easy to verify that
$$(x_1,x_2,x_3,x_4) \in \lambda \iff  (x_2,x_4,x_1,x_3) \in \Delta_2^4$$
 and thus $\pp(\Delta_2^4)= \pp (\lambda)$. Note that $\pp(\Delta_1^2)= \pp(\le)$ is a maximal partial clone on $\2$.
 \end{example}

 The following result  will be used to show that the family  $\pp (\F^{(h)})_{h \ge 2}$ is a descending chain of partial clones containing all partial constant functions.

 \begin{lemma}\label{lem:decreasing} Let $1 \le t < h$. Then
 
  \centerline{$\displaystyle  \bigcap_{j=1}^{h} \pp (\Delta_j^{h+1})
  \subseteq \pp(\Delta_t^h)$.}
 \end{lemma}
 \begin{proof}
  Let  $f \in \bigcap_{j=1}^{h} \pp (\Delta_j^{h+1})$ be an $n$-ary partial function on $\2$ and fix  $t \in \{1,\dots, h-1\}$.  
  We show that $$f \in \pp(\Delta_t^h).$$

 To do so take an $h \times n$ matrix $M$  with all columns in $\Delta_t^h$ and all rows in $\d(f)$.  Now let  $N$ be the $(h+1) \times n$ matrix whose rows satisfy
 $N_{i*} = M_{i*}$ for $i=1,\dots, h$ and $N_{(h+1)*}=M_{h*}$,  i.e., $N$ is obtained from $M$ by duplicating the last row of $M$. Since all columns of $M$ belong to $\Delta_t^h$, the $h$-tuple $\vv_t^h$ is not a column of $M$ and consequently the $(h+1)$-tuple $\vv_t^{h+1}$ is not a column of $N$. Thus all columns of $N$ belong to $\Delta_t^{h+1}$. As $f \in \pp (\Delta_t^{h+1})$ we have that the $(h+1)$-tuple $(f(N_{1*}),\ldots,f(N_{h*}),f(N_{h*})) \in \Delta_t^{h+1}$, i.e.,
 $(f(N_{1*}),\ldots,f(N_{h*}),f(N_{(h+1)*})) \ne \vv_t^{h+1}$ and so
$(f(M_{1*}),\ldots,f(M_{h*})) \ne \vv_t^h$, proving that $f \in \pp(\Delta_t^h)$.
 \end{proof}

 \begin{remark} 
The above result can be shown using the representation lemma due to B. Romov (see \cite{lau2006} Lemma 20.3.4). Indeed for $1<t<h$ we have
$$\Delta_t^h = \{ (x_1,\dots,x_h) \mid (x_1,\dots,x_h,x_h) \in \Delta_t^{h+1} \},$$
and so by  Lemma 20.3.4 of \cite{lau2006} we obtain  $\pp (\Delta_t^{h+1}) \subseteq \pp (\Delta_t^{h})$. Thus
$\bigcap_{j=1}^{h} \pp (\Delta_j^{h+1}) \subseteq \pp (\Delta_t^{h}).$
 \end{remark}

From  Lemma \ref {lem:decreasing}, we get:

\begin{corollary}\label{corollary}  Let $ h \ge 2$. Then
 $\pp(\F^{(h+1)}) \subseteq \pp(\F^{(h)})$. 
 \end{corollary}

  The above inclusion is strict. Indeed consider the partial function $\f_n$ defined in Lemma \ref{phin} where $n:=h+1$. Then $\f_n$ preserves  
  $\Delta_t^h$ for all $t=1,\dots, h-1$ but $\f_n$ does not preserve $\Delta_1^{h+1}$.  

  \begin{corollary}
  $\pp(\F^{(2)}) \supset \pp(\F^{(3)}) \supset \pp(\F^{(4)}) \dots$. 
  \end{corollary}

 We now show that the limit of the above chain is the partial clone generated by the constant functions on $\2$.

 \begin{lemma}\label{lastresult}
 Let $n \ge 2$ and $f$ be an $n$-ary partial function that is neither a partial projection nor a partial constant function on $\2$. Then there is an $h \ge 2$ such that $f \not\in \pp(\F^{(h)})$.
 \end{lemma}
 \begin{proof}
  Set $h:=|\d f |$ and form an $h \times n$ matrix $M$ whose rows consists of all tuples in the domain of $f$ and consider the $h$-tuple $f(M):= (f(M_{1*}), \dots, f(M_{h*}))$. 
  Since $f$ is not a partial constant function, the tuple $f(M)$ contains at least one symbol 0 and one symbol 1. Let $t$ be the number of symbols 1 in the tuple $f(M)$. 
  Call $N$ the $h \times n$ matrix obtained by rearranging the rows of the matrix $M$ so that  $f(N):= (f(N_{1*}), \dots, f(N_{h*}))=\vv_t^h$. 
  Since $f$ is not a partial projection function, no column of $N$ is the tuple $\vv_t^h$. Thus all columns of $N$ belong to the $h$-ary relation $\Delta_t^h$, 
  while  $(f(N_{1*}), \dots, f(N_{h*}))\not\in \Delta_t^h$, proving that  $f \not\in \pp(\Delta_t^h)$. Thus $f \not\in \pp(\F^{(h)})$. 
 \end{proof}

\medskip
 By combining  these results, we get:\\
 \begin{theorem} $\displaystyle {\mathcal C} = \bigcap_{ h \ge 2}\pp(\F^{(h)})$.
  \end{theorem}
  
 The above theorem shows that the family $\bigcup\{\F^{(h)} \st {h \ge 2}\}$ is a hereditarily strongly rigid family over $\2$.

\begin{remark}
 It is not hard to construct infinite subfamilies of $\{\F^{(h)} \st {h \ge 2}\}$  of relations that are hereditarily strongly rigid. For example consider the family of relations
$\{ \Delta_{n}^{2n} \mid n \in \N \}$. Then by the Romov representation lemma we have $\pp (\Delta_1^2) \supseteq \pp(\Delta_2^4) \supseteq \pp(\Delta_3^6) \dots$ Moreover, similar arguments as in Corollary \ref{corollary} and  Lemma \ref{lastresult} give that the above inclusions are all strict and
$\displaystyle {\mathcal C} = \bigcap_{ n \ge 2}\pp(\Delta_{n}^{2n}).$
\end{remark}

\section*{Acknowledgments} {This research was completed  while the second  author was visiting the Camille Jordan Institute in Lyon  in June  2014}. The authors wish to thank the reviewers of this paper, in particular an anonymous reviewer for a very detailed reading, pointing some flaws and making numerous useful suggestions.


\begin{thebibliography}{4}

\bibitem  {miyakawa2} C.Delhomm\'e, M.Miyakawa,  M. Pouzet,  I.G. Rosenberg and H.Tatsumi, Semirigid  system of three equivalence relations, IEEE Proceedings of ISMVL-2012, May 14-17, Victoria, BC, Canada,  6pp.



\bibitem{anne} A. Fearnley, A strongly rigid binary relation, {\it Acta Sci. Math.} Szeged {\bf 61} (1995), pp 35-41.



\bibitem{engel} K. Engel, Sperner theory. Encyclopedia of Mathematics and its Applications, 65. Cambridge University Press, Cambridge, 1997. x+417 pp.

\bibitem{lan-pos} H. L\"anger and R. P\"oschel, Relational systems with trivial endomorphisms and polymorphisms. {\it Jour. of Pure and Applied Algebra}, {\bf 32} (1984) pp 129 - 142.

\bibitem{benoit-claude} B. Larose and C. Tardif, Strongly rigid graphs and projectivity. {\it Mult.-Valued Log.}, {\bf 7}, no. 5-6 (2001), pp 339-361.


\bibitem  {miyakawa1} M. Miyakawa, M.Pouzet,  I.G. Rosenberg and H.Tatsumi, Semirigid equivalence relations on a finite set,
{\it J. Mult.-Valued Logic Soft Comput.}, {\boldmath 15}(4), (2009),
   395--407.



\bibitem{romov1} B. A. Romov, Maximal subalgebras of algebras of
partial multivalued logic functions, {\it Kibernatika}; English
translation in {\it Cybernetics} {\bf 1} (1980) pp. 31-41.

 \bibitem{lau2006}  D. Lau,  Function algebras on finite sets. A basic
 course on many-valued logic and clone theory. \emph{ Springer Monographs
 in Mathematics} (2006)



\bibitem{ivo} I. G. Rosenberg, Strongly rigid relations, {\it Rocky Mountain J. Math} (1973), pp 631-639.

\bibitem{v-p-h} P. Vop\u{e}nka, A. Pultr and Z. Hedr\'{\i}n, A rigid relation exists on any set, {\it Comment. Math. Univ. Carolinae} 6, (1965) pp 149-155.

\bibitem {zadori} L. Z\'adori L., Generation of finite partition
 lattices, {\it Lectures in Universal Algebra} (Proc. Colloq., Szeged,
1983),  Colloq. Math. Soc. J\'anos Bolyai 43, North-Holland,
Amsterdam, 1986, 573-586.


\end{thebibliography}
 \end{document}